\documentclass[11pt,a4paper]{amsart}
\usepackage{amssymb,amsthm,enumerate}
\usepackage{graphicx}
\usepackage{subcaption}
\usepackage{cite}
\usepackage[myheadings]{fullpage}
\usepackage{hyperref}
\hypersetup{
	colorlinks=true,   
	linkcolor=blue,
	citecolor=blue,
}
\hypersetup{
	colorlinks=true,   
	linkcolor=blue,
	citecolor=blue,
}

\newtheorem{cor*}{Corollary}
\newtheorem{prop*}{Proposition}
\newtheorem{thm*}{Theorem}

\newtheorem{theorem}{Theorem}[section]

\newtheorem{prop}[theorem]{Proposition}

\theoremstyle{definition}

\newtheorem{exmp}[theorem]{Example}

\newcommand{\F}{\mathcal{F}}
\newcommand{\G}{\mathcal{G}}
\newcommand{\Z}{\mathbb{Z}}
\newcommand{\N}{\mathbb{N}}
\newcommand{\C}{\mathcal{C}}
\newcommand{\h}{\mathcal{H}}
\newcommand{\BS}{\mathrm{BS}}
\DeclareMathOperator{\map}{Mod}
\DeclareMathOperator{\Diff}{Diff^{+}}

\everymath{\displaystyle}

\makeatletter
\@namedef{subjclassname@2020}{\textup{2020} Mathematics Subject Classification}
\makeatother

\begin{document}
\emergencystretch=3em
\hbadness=3000
\title[Baumslag-Solitar subgroups of the mapping class group]{Baumslag-Solitar subgroups of \\ the mapping class group}
	
\author{Pankaj Kapari}
\address{Department of Mathematics\\
Tata Institute of Fundamental Research Mumbai\\
A-431, School of Mathematics, TIFR, Old Navy Nagar, Colaba\\ Mumbai 400005, Maharashtra\\
India}
\email{pankajkapri02@gmail.com}
\urladdr{https://sites.google.com/view/pankajkapdi/home}

\author{Dev Krishna}
\address{Department of Mathematics\\
Indian Institute of Science Education and Research Bhopal\\
Bhopal Bypass Road, Bhauri \\
Bhopal 462066, Madhya Pradesh\\
India}
\email{dev22@iiserb.ac.in}
	
\author{Kashyap Rajeevsarathy}
\address{Department of Mathematics\\
Indian Institute of Science Education and Research Bhopal\\
Bhopal Bypass Road, Bhauri \\
Bhopal 462066, Madhya Pradesh\\
India}
\email{kashyap@iiserb.ac.in}
\urladdr{https://home.iiserb.ac.in/\string~kashyap/}

	
	
\begin{abstract}
For $g\geq 2$ and nonzero integers $p,q$, let $\mathrm{Mod}(S_g)$ be the mapping class group of a closed oriented surface $S_g$ of genus $g$, and let $\mathrm{BS}(p,q)$ be the Baumslag-Solitar group. We provide necessary and sufficient conditions under which two mapping classes $G,F\in \mathrm{Mod}(S_g)$ generate a subgroup isomorphic to $\mathrm{BS}(p,q)$. In particular, if $\mathrm{BS}(p,q)$ embeds in $\mathrm{Mod}(S_g)$, then $|p|=|q|$ and $G,F$ are reducible mapping classes of infinite order. We also construct subgroups isomorphic to $\mathrm{BS}(p,p)$ and $\mathrm{BS}(p,-p)$ for $p>1$. Finally, we show that every infinite metacyclic subgroup and every Baumslag-Solitar subgroup of $\mathrm{Mod}(S_g)$ lifts to an isomorphic subgroup of $\mathrm{Diff}^+(S_g)$, that is, these subgroups satisfy the generalized Nielsen realization.
\end{abstract}

\maketitle
\section{Introduction}
For $g\geq 2$, let $S_g$ be a connected, closed, and oriented surface of genus $g$. Let $\Diff(S_g)$ be the group of orientation-preserving homeomorphisms of $S_g$ with the compact-open topology. The \textit{mapping class group} $\map(S_g)$ of $S_g$ is defined by $\pi_0(\Diff(S_g))$, that is, it is the group of isotopy classes (or path components) of elements of $\Diff(S_g)$. For nonzero integers $p,q$, the \textit{Baumslag-Solitar group} is given by a presentation $\BS(p,q)=\langle G, F \mid G F^p G^{-1} = F^q \rangle$. Given $G,F\in \map(S_g)$, it is natural to ask the following question: can one derive necessary and sufficient conditions under which $F$ and $G$ generate a subgroup of $\map(S_g)$ isomorphic to $\BS(p,q)$? We note that $\BS(1,\pm 1)$ is an infinite metacyclic group (cyclic extension of a cyclic group), and such necessary and sufficient conditions were derived in~\cite{inf_meta_ggd} for this case. For $p>1$, Kent-Leininger~\cite{leininger_survey} constructed a subgroup of $\map(S_g)$ isomorphic to $\BS(p,p)$. Taking inspiration from these works, in this paper, we derive such conditions for non-metacyclic $\BS(p,q)$ and give some constructive examples of such subgroups of $\map(S_g)$.

A simple closed curve is said to be \textit{essential} if it is not null-homotopic. A collection of isotopy classes of pairwise disjoint essential simple closed curves on $S_g$ is called a \textit{multicurve} on $S_g$. The elements of $\map(S_g)$ are called \textit{mapping classes}. A mapping class $F\in \map(S_g)$ is said to be \textit{periodic} if the order of $F$ is finite and \textit{reducible} if $F$ preserves a multicurve on $S_g$ (this multicurve is called a \textit{reduction system} for $F$). A mapping class is called \textit{irreducible} if it is not reducible. Due to Nielsen-Thurston classification~\cite{thurston_bams_1988}, irreducible mapping classes of infinite order are called \textit{pseudo-Anosov} mapping classes. For a reducible mapping class $F$, the intersection of all maximal reduction systems of $F$ is known as the \textit{canonical reduction system} for $F$, and it will be denoted by $\C(F)$. For a simple closed curve $c$, let $T_c$ denote the left-handed Dehn twist about $c$. A product of powers of Dehn twists about curves in a multicurve $C$ is called a \textit{multitwist} about $C$.

Let $F\in \map(S_g)$ be an infinite-order reducible mapping class whose canonical reduction system is $\C(F)=\{c_1,c_2,\dots,c_k\}$. Let $N$ be an $F$-invariant closed regular neighborhood of $\mathcal{C}(F)$. There exists a least positive integer $n$ such that $F^n$ fixes each path component of $\overline{S_g\setminus N}$ and each curve in $\C(F)$, and
\begin{equation}
\label{eq:nt_decomp}
F^n=T_{c_1}^{n_1}T_{c_2}^{n_2}\cdots T_{c_k}^{n_k}\eta_1(F_1)\eta_2(F_2)\cdots \eta_k(F_k),
\end{equation}
where each $F_i \in \map(R_i)$ is either the identity or pseudo-Anosov, $n_i\in \Z$, $R_i$ is a path component of $S_g(\C(F))$, and $\eta_i:\map(R_i) \rightarrow \map(S_g)$ is the natural inclusion map. The product $T_{c_1}^{n_1}T_{c_2}^{n_2}\cdots T_{c_k}^{n_k}$ appearing in~\eqref{eq:nt_decomp} is the \textit{multitwist component} of $F$. The surface obtained by capping the boundary components of $\overline{S_g\setminus N}$ by marked disks will be denoted by $S_g(\C(F))$. A mapping class $F\in \map(S_g)$ is said to be \textit{pseudo-periodic} if some power of $F$ is a multitwist.

When $p,q>0$, Kent-Leininger~\cite{leininger_survey} proved that if $\BS(p,q)$ embeds in $\map(S_g)$ then $p=q$. In this direction, we proved the following result (see Proposition~\ref{prop:bspq_in_mod}).

\begin{prop*}
\label{prop1}
For $g\geq 2$, $p,q\in \Z\setminus \{0\}$, and $F,G\in \map(S_g)$, let
\[
\BS(p,q)=\langle F,G \mid GF^pG^{-1}=F^q\rangle
\]
be a subgroup of $\map(S_g)$. Then:
\begin{enumerate}[(i)]
\item $|p|=|q|$ and
\item $F$, $G$ are reducible mapping classes of infinite order.
\end{enumerate}
\end{prop*}

From now on, we assume that $p>1$ and $q=p\epsilon$, where $\epsilon=\pm 1$. Given $G,F\in \map(S_g)$, the following result provides necessary and sufficient conditions under which $F$ and $G$ generate a subgroup of $\map(S_g)$ isomorphic to $\BS(p,p\epsilon)$ such that $\langle F^p\rangle \triangleleft \langle F,G\rangle$, where $\epsilon=\pm 1$ (see Theorem~\ref{thm:main_theorem}).

\begin{thm*}
\label{thm1}
For $g\geq 2$, let $F,G \in \map(S_g)$ be nontrivial reducible mapping classes of infinite order. Let
\begin{center}
$T_{c_1}^{n_1}T_{c_2}^{n_2}\cdots T_{c_k}^{n_k}$ and $T_{d_1}^{m_1}T_{d_2}^{m_2}\cdots T_{d_{\ell}}^{m_{\ell}}$
\end{center}
be the multitwist component of $F$ and $G$, $\C(F)=\{c_1,c_2,\dots,c_k\}$ and $\C(G)=\{d_1,d_2,\dots,d_{\ell}\}$ are the canonical reduction system of $F$ and $G$, respectively, where $n_i,m_j\in \Z$ for $1\leq i\leq k$ and $1\leq j\leq \ell$. For $\epsilon=\pm 1$ and $p>1$, we have
\[
\BS(p,p\epsilon)=\langle F,G \mid GF^pG^{-1}=F^{p\epsilon} \rangle
\]
if and only if the following conditions hold:
\begin{enumerate}[(i)]
\item $\mathcal{C}(F) \cup \mathcal{C}(G)$ is a multicurve.
\item $G(A_i) = B_i$, $G(B_i) = A_i$, and $F^p(C_i) = C_i$ for every $i$, where
\begin{align*}
A_i&=\{c_j\in\C(F)\mid n_j=n_i\},\\
B_i&=\{c_j\in\C(F)\mid n_j=\epsilon n_i\},\\
C_i&=\{d_j\in\C(G)\mid m_j=m_i\}.
\end{align*}
\item For every path component $R$ of $S_g(\mathcal{C}(F))$, we have
\[
\langle F_r,G_r\mid G_rF_r^pG_r^{-1}=F_r^{p\epsilon^{q_r}}, F_r^{u_r}=G_r^{v_r}=1 \rangle,
\]
where $F_r,G_r\in \map(R)$ are induced by $F,G$, $u_r,v_r\in \N\cup \{\infty\}$ are orders of $F,G$, respectively, $q_r\in \N$ is the size of the orbit of $R$ under $G$.
\item For two path components $R$ and $S$ of $S_g(\mathcal{C}(F))$ such that $G(R) = S$, we have $F_r^p$ is conjugate to $F_s^{p\epsilon}$, where $F_r \in \map(R)$ and $F_s \in \map(S)$ are induced by $F$.
\end{enumerate}
\end{thm*}

In Section~\ref{sec:bsmn_mod}, we provide some examples of Baumslag-Solitar subgroups of $\map(S_g)$ whose constructions differ from that of Kent-Leininger (see Examples~\ref{exmp3} - \ref{exmp4}). Let $\pi:\Diff(S_g)\to \map(S_g)$ be the natural projection map. Due to Kerckhoff~\cite{kerckhoff_annals_1983}, finite subgroups of $\map(S_g)$ can be realized as a group of hyperbolic isometries for some hyperbolic metric on $S_g$. Since the isometry group of a closed hyperbolic surface is finite~\cite[Chapter 7]{primer}, infinite subgroups of $\map(S_g)$ cannot be realized as groups of hyperbolic isometries. For infinite subgroups of $\map(S_g)$, one can ask if they can be realized as a subgroup of $\Diff(S_g)$, that is, does an infinite subgroup $H$ of $\map(S_g)$ lift to an isomorphic subgroup of $\Diff(S_g)$ under the map $\pi$. In the literature, this is called the \textit{section problem} or the \textit{generalized Nielsen realization problem}~\cite[Section 6.3]{farb_problems_book} (also see~\cite{mann_gnr}). It is known that~\cite{mark} $\map(S_g)$ does not lift to $\mathrm{Homeo}^+(S_g)$ under the natural map $\mathrm{Homeo}^+(S_g) \to \map(S_g)$. In contrast, free subgroups and free abelian subgroups of $\map(S_g)$ can be lifted to $\Diff(S_g)$. In this direction, we have the following result (see Theorems~\ref{thm:nr_inf_meta} - \ref{thm:nr_bsmn}).

\begin{thm*}
\label{thm2}
For $g\geq 2$, let $H$ be either an infinite metacyclic subgroup or a Baumslag-Solitar subgroup of $\map(S_g)$. Then there exists a subgroup $\tilde{H}$ of $\Diff(S_g)$ such that $\tilde{H}\xrightarrow{\pi} H$ is an isomorphism. 
\end{thm*}

This paper is organized as follows. In Section~\ref{sec:bsmn_mod}, we study Baumslag-Solitar subgroups of $\map(S_g)$ and prove Proposition~\ref{prop1} and Theorem~\ref{thm1}. Furthermore, we construct some Baumslag-Solitar subgroups of $\map(S_g)$ (see Examples~\ref{exmp1} - \ref{exmp4}). We conclude the paper by proving Theorem~\ref{thm2} in Section~\ref{sec:gnr_bsmn}.

\section{Baumslag-Solitar subgroups of the mapping class group}
\label{sec:bsmn_mod}
In this section, first we determine when $\BS(p,q)$ embeds in $\map(S_g)$ and after that we will provide necessary and sufficient conditions under which two mapping classes generate a Baumslag-Solitar subgroup of $\map(S_g)$. We will also construct some explicit Baumslag-Solitar subgroups of $\map(S_g)$.

The proof of part~(i) of the following proposition in the case when $p,q>0$ is due to Kent-Leininger~\cite[Lemma 8.2]{leininger_survey}. A similar argument works when $p>0$ and $q<0$. We are including a proof for the sake of completion.

\begin{prop}
\label{prop:bspq_in_mod}
For $g\geq 2$, $p,q\in \Z\setminus \{0\}$, and $F,G\in \map(S_g)$, let
\[
\BS(p,q)=\langle F,G \mid GF^pG^{-1}=F^q\rangle
\]
be a subgroup of $\map(S_g)$. Then:
\begin{enumerate}[(i)]
\item $|p|=|q|$ and
\item $F$, $G$ are reducible mapping classes of infinite order.
\end{enumerate}
\end{prop}

\begin{proof}
Since $\BS(p,q)$ is torsion-free, $F$ and $G$ are not periodic. Without loss of generality, we assume that $p>0$. First, assume that $F$ is pseudo-Anosov with the stretch factor $\lambda>1$. We observe that the stretch factor of $F^p$ and its conjugate $GF^pG^{-1}$ is $\lambda^p$. The stretch factor of $F^q$ is $\lambda^q$ when $q>0$ and $\lambda^{-q}$ when $q<0$. Since $GF^pG^{-1}=F^q$, we have $\lambda^p=\lambda^q$ when $q>0$ and $\lambda^p=\lambda^{-q}$ when $q<0$. It follows that either $p=q$ or $p=-q$, that is, $|p|=|q|$.

Now, assume that $F$ is a reducible mapping class of infinite order with multitwist component $T_{c_1}^{n_1}T_{c_2}^{n_2}\cdots T_{c_k}^{n_k}$. Comparing the multitwist component of $GF^pG^{-1}$ and $F^q$, we get
\[
T_{G(c_1)}^{pn_1}T_{G(c_2)}^{pn_2}\cdots T_{G(c_k)}^{pn_k}=T_{c_1}^{qn_1}T_{c_2}^{qn_2}\cdots T_{c_k}^{qn_k}.
\]
Depending on signs of $n_i$'s, it follows that either $p=q$ or $p=-q$, that is, $|p|=|q|$.

We observe that the subgroup $\langle F^p, G^2\rangle$ of $\BS(p,\pm p)$ is isomorphic to $\Z^2$. It follows from~\cite[Theorem 3.1]{inf_meta_ggd} that $F$ and $G$ cannot be pseudo-Anosov mapping classes. Thus, $F$ and $G$ are reducible mapping classes of infinite order.
\end{proof}
    
\begin{theorem}
\label{thm:main_theorem}
For $g\geq 2$, let $F,G \in \map(S_g)$ be nontrivial reducible mapping classes of infinite order. Let
\begin{center}
$T_{c_1}^{n_1}T_{c_2}^{n_2}\cdots T_{c_k}^{n_k}$ and $T_{d_1}^{m_1}T_{d_2}^{m_2}\cdots T_{d_{\ell}}^{m_{\ell}}$
\end{center}
be the multitwist component of $F$ and $G$, $\C(F)=\{c_1,c_2,\dots,c_k\}$ and $\C(G)=\{d_1,d_2,\dots,d_{\ell}\}$ are the canonical reduction system of $F$ and $G$, respectively, where $n_i,m_j\in \Z$ for $1\leq i\leq k$ and $1\leq j\leq \ell$. For $\epsilon=\pm 1$ and $p>1$, we have
\[
\BS(p,p\epsilon)=\langle F,G \mid GF^pG^{-1}=F^{p\epsilon} \rangle
\]
if and only if the following conditions hold:
\begin{enumerate}[(i)]
\item $\mathcal{C}(F) \cup \mathcal{C}(G)$ is a multicurve.
\item $G(A_i) = B_i$, $G(B_i) = A_i$, and $F^p(C_i) = C_i$ for every $i$, where
\begin{align*}
A_i&=\{c_j\in\C(F)\mid n_j=n_i\},\\
B_i&=\{c_j\in\C(F)\mid n_j=\epsilon n_i\},\\
C_i&=\{d_j\in\C(G)\mid m_j=m_i\}.
\end{align*}
\item For every path component $R$ of $S_g(\mathcal{C}(F))$, we have
\[
\langle F_r,G_r\mid G_rF_r^pG_r^{-1}=F_r^{p\epsilon^{q_r}}, F_r^{u_r}=G_r^{v_r}=1 \rangle,
\]
where $F_r,G_r\in \map(R)$ are induced by $F,G$, $u_r,v_r\in \N\cup \{\infty\}$ are orders of $F,G$, respectively, $q_r\in \N$ is the size of the orbit of $R$ under $G$.
\item For two path components $R$ and $S$ of $S_g(\mathcal{C}(F))$ such that $G(R) = S$, we have $F_r^p$ is conjugate to $F_s^{p\epsilon}$, where $F_r \in \map(R)$ and $F_s \in \map(S)$ are induced by $F$.
\end{enumerate}
\end{theorem}

\begin{proof}
For $\epsilon=\pm 1$ and $p>1$, assume that
\[
\BS(p,\epsilon p)=\langle F,G \mid GF^pG^{-1}=F^{\epsilon p} \rangle.
\]
Since $\langle G^2,F^p\rangle \cong \Z^2$, conditions~(i)-(ii) follow from~\cite[Theorem 3.12]{inf_meta_ggd}. Since $G(\C(F))=\C(F)$, $\C(F)$ is the canonical reduction system for $\BS(p,\epsilon p)$. Let $R$ be a path component of $S_g(\C(F))$ of orbit size $q_r$. For $F_r:=F^{q_r}_{|_R},G_r:=G^{q_r}_{|_R}\in \map(R)$, restricting the relation $G^{q_r}F^{q_rp}G^{-q_r}=F^{q_rp\epsilon^{q_r}}$ in $\map(R)$, we obtain~(iii) (note that if either $F_r$ or $G_r$ is of infinite order, we have $u_r$ or $v_r$ equals $\infty$). Let $R$ and $S$ be two path components of $S_g(\C(F))$ such that $G(R)=S$. Since $G(R)=S$, restricting the relation $GF^{q_rp}G^{-1}=F^{q_rp\epsilon}$ to $S$, we get that $F_r^p$ is conjugate to $F_s^{p\epsilon}$.

Conversely, assume that $F$ and $G$ satisfy $(i)-(iv)$. Since the groups generated by the induced maps of $F$ and $G$ on each path component $S_g(\C(F))$ satisfy appropriate presentations (as in $(iii)$), it follows that $F$ and $G$ present the Baumslag-Solitar group $\BS(p,p\epsilon)$.  
\end{proof}

The construction in Example~\ref{exmp1} is due to Kent-Leininger (see~\cite[page 21]{leininger_survey}).
\begin{exmp}
\label{exmp1}
For $g\geq 2$, let $c$ be an essential simple closed curve on $S_g$. Let $F'\in \map(S_g)$ be a periodic mapping class of order $p$ and $G$ be a partial pseudo-Anosov both supported on $S_g (\{c\})$. Due to Min~\cite[Corollary 4.3]{min}, we can assume that the subgroup $\langle G,F'\rangle\cong\Z\ast\Z_p$. For $F=F'T_c$, we have that
\[
GF^pG^{-1}=GT_c^pG^{-1}=T_{G(c)}^p=T_c^p=F^p.
\]
By Theorem~\ref{thm:main_theorem}, it follows that $\langle G,F\rangle \cong \BS(p,p)$.
\end{exmp}

We provide a similar construction as in Example~\ref{exmp1} to construct a subgroup of $\map(S_g)$ isomorphic to $\BS(p,-p)$.

\begin{exmp}
\label{exmp2}
Let $c_1$ and $c_2$ be essential simple closed curves as in Figure~\ref{fig:leininger2}. Let $F_1$ be a periodic mapping class of order $p$ and $G'$ be a partial pseudo-Anosov both supported on $S_g(\{c_1,c_2\})$ and preserve curves $c_1$ and $c_2$. Let $F_2$ be an involution represented by $\pi$-rotation of $S_g$ such that $F_2(c_1)=c_2$ and $F_2(c_2)=c_1$ (see Figure~\ref{fig:leininger2}). Due to Min~\cite[Corollary 4.3]{min}, we can assume that $\langle F_2G',F_1\rangle\cong\Z\ast\Z_p$. For $G=F_2G'$ and $F=F_1T_{c_1}T_{c_2}^{-1}$, we observe that $F(c_i)=c_i$, $G(c_1)=c_2$, and $G(c_2)=c_1$, where $i=1,2$. Therefore, we have that
\[
GF^pG^{-1}=GT_{c_1}^pT_{c_2}^{-p}G^{-1}=T_{G(c_1)}^pT_{G(c_2)}^{-p}=T_{c_2}^pT_{c_1}^{-p}=F^{-p}.
\]
By Theorem~\ref{thm:main_theorem}, it follows that $\langle G,F \rangle\cong \BS(p,-p)$.
\begin{figure}[htbp]
\centering
\begin{subfigure}[t]{0.45\textwidth}
\centering
\includegraphics[width=\linewidth]{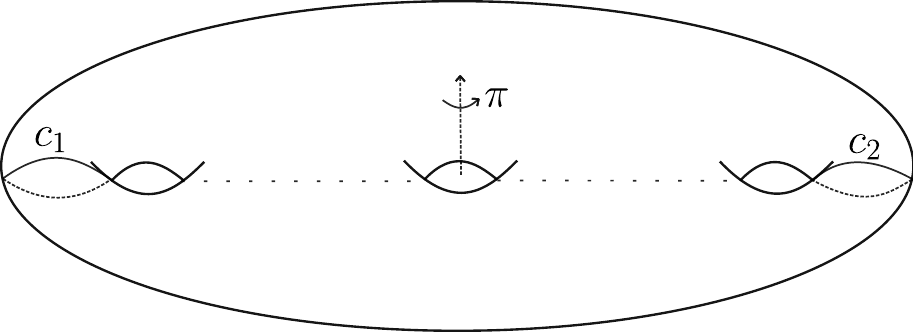}
\caption{When $g$ is odd.}
\end{subfigure}
\hfill
\begin{subfigure}[t]{0.45\textwidth}
\centering
\includegraphics[width=\linewidth]{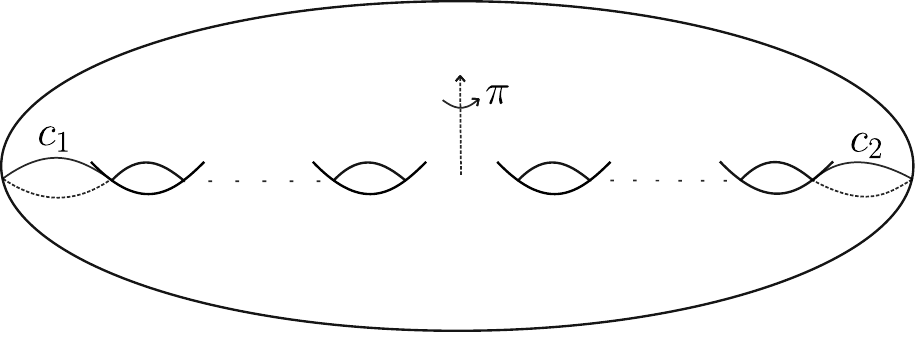}
\caption{When $g$ is even.}
\end{subfigure}
\caption{A pair of simple closed curves on $S_g$ permuted by an involution.}
\label{fig:leininger2}
\end{figure}
\end{exmp}

For isotopy classes of two essential simple closed curves $a$ and $b$, the \textit{geometric intersection number} $i(a,b)$ between $a$ and $b$ is the minimal transverse intersection between any representatives of $a$ and $b$. For our next example, we need the following result. 

\begin{theorem}[{Humphries~\cite[Theorem 1.1]{humphries_gms_1989}}]
\label{thm:humphries}
For $g\geq 2$, let $C=\{c_1,c_2,\dots,c_k\}$ be a collection of distinct essential simple closed curves on $S_g$. If $i(c_i,c_j)>1$ for all $i \neq j$ and no component of $S_g \setminus C$ is a disc, then $\langle T_{c_1},T_{c_2},\dots,T_{c_k} \rangle$ is a free group of rank $k$.
\end{theorem}

In the examples constructed above, the conjugating generator (i.e., $G$) of $\BS(p,p\epsilon)$ is a partial pseudo-Anosov. In the following example, both the generators are pseudo-periodic.

\begin{exmp}
\label{exmp3}
Let $a$ and $b$ be simple closed curves on $S_4$ and $H$ be a periodic mapping class of order $3$ represented by $2\pi/3$-rotation of $S_4$ as in Figure~\ref{fig:triangle}. We claim that $\langle T_a,H\rangle \cong \Z\ast \Z_3$. Suppose we have a word $w=T_{a}^{n_1}H^{m_1}T_{a}^{n_2}H^{m_2}\cdots T_{a}^{n_k}H^{m_k}\in \langle T_a,H\rangle$, where $n_i\in \Z$ and $m_k=1,2$. We have
\begin{align*}
w&=T_{a}^{n_1}H^{m_1}T_{a}^{n_2}H^{m_2}\cdots T_{a}^{n_k}H^{m_k}\\
&=T_{a}^{n_1}T_{H^{m_1}(a)}^{n_2}H^{m_1+m_2}T_{a}^{n_3}H^{m_3}\cdots T_{a}^{n_k}H^{m_k}\\
&=T_{a}^{n_1}T_{H^{m_1}(a)}^{n_2}T_{H^{m_1+m_2}(a)}^{n_3}H^{m_1+m_2+m_3}\cdots T_{a}^{n_k}H^{m_k}\\
&=\cdots\\
&=T_{a}^{n_1}T_{H^{m_1}(a)}^{n_2}T_{H^{m_1+m_2}(a)}^{n_3}\cdots T_{H^{\sum_{i=1}^{k-1}m_i}(a)}^{n_k}H^{\sum_{i=1}^k m_i}.
\end{align*}
By Theorem~\ref{thm:humphries}, since no path component of $S_4(\{a,H(a),H^2(a)\})$ is a disc, it follows that $\langle T_a,T_{H(a)},T_{H^2(a)}\rangle$ is a free group of rank $3$. The claim follows as $w\neq 1$. For $G=T_a$ and $F=HT_b$, since
\[
GF^3G^{-1}=GT_b^3G^{-1}=T_{G(b)}^3=T_b^3=F^3,
\]
it follows from Theorem~\ref{thm:main_theorem} that $\langle G, F\rangle\cong \BS(3,3)$.
\begin{figure}[htbp]
\centering
\includegraphics[width=.38\textwidth]{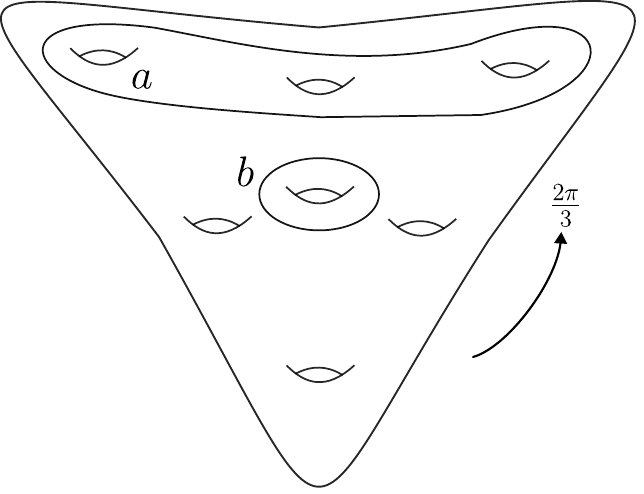}
\caption{A collection $\{a,H(a),H^2(a)\}$ of simple closed curves whose Dehn twists generate a free group.}
\label{fig:triangle}
\end{figure}
\end{exmp}

In the following example, $G$ will be pseudo-periodic, and $F$ will be a partial pseudo-Anosov.

\begin{exmp}
\label{exmp4}
Consider the curves on $S_6$ and let $H$ be the involution as shown in Figure~\ref{fig4}. Let $F=T_{a_1}T_{b_1}^{-1}H$ and $G=T_cH$. We note that $F$ is a partial pseudo-Anosov mapping class. We have $F^2=T_{a_1}T_{b_1}^{-1}HT_{a_1}T_{b_1}^{-1}H^{-1}=T_{a_1}T_{b_1}^{-1}T_{a_2}T_{b_2}^{-1}$ and $GF^2G^{-1}=F^2$. By Theorem~\ref{thm:humphries}, $\langle T_c,T_{H(c)}\rangle$ is a free group of rank $2$. Let $R$ be the path component of $S_6(\C(F))$ of genus $4$. A similar argument as in Example~\ref{exmp3} implies that $\langle F_r,G_r\rangle=\langle T_c,H\rangle\cong \Z\ast \Z_2$. Thus, by Theorem~\ref{thm:main_theorem} $\langle G,F\rangle\cong \BS(2,2)$.
\begin{figure}[htbp]
\centering
\includegraphics[width=.65\textwidth]{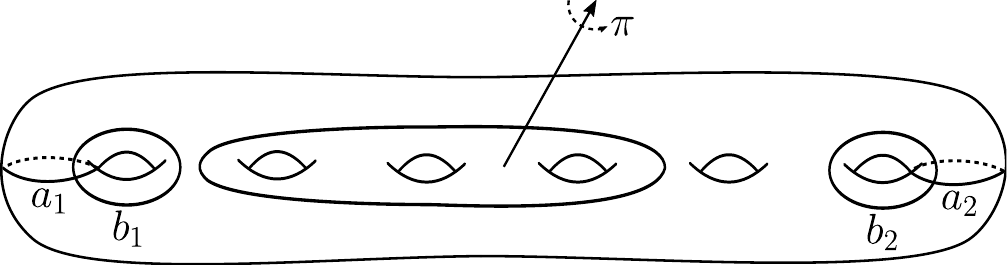}
\caption{A collection of curves on $S_6$.}
\label{fig4}
\end{figure}
\end{exmp}

\section{Towards the generalized Nielsen realization problem}
\label{sec:gnr_bsmn}
In this section, we show that infinite metacyclic and Baumslag-Solitar subgroups of $\map(S_g)$ lift to an isomorphic subgroup of $\Diff(S_g)$. Before proving the main results of this section, we state some results from the literature that will be used in our arguments.

\begin{theorem}[{Bonahon~\cite[Proposition 4.3]{bonahon_cobordism_1983}}]
\label{thm:bonahon_flp}
For $g\geq 2$, let $\F,\G\in \Diff(S_g)$ be either periodic or pseudo-Anosov diffeomorphisms such that $\F$ is isotopic to $\G$. Then there exists $\h\in \Diff(S_g)$ isotopic to the identity such that $\F=\h\G\h^{-1}$.
\end{theorem}

\begin{theorem}[{Chilina~\cite[Theorem 1]{chilina_2025}}]
\label{thm:chilina}
For $g\geq 2$, let $\F\in \Diff(S_g)$ be a pseudo-Anosov diffeomorphism. For $\G\in \Diff(S_g)$, if $\G\F=\F\G$ then $\G$ is either periodic or pseudo-Anosov. 
\end{theorem}

We first prove the generalized Nielsen realization for the infinite metacyclic subgroups of $\map(S_g)$.
   
\begin{theorem}
\label{thm:nr_inf_meta}
For $g\geq 2$ and $G,F \in \map(S_g)$, let $\langle G,F \rangle$ be a non-cyclic infinite metacyclic subgroup of $\map(S_g)$ with $\langle F \rangle \triangleleft \langle G,F\rangle$. Then, there exist $\G,\F\in \Diff(S_g)$ such that the natural projection map $\pi:\Diff(S_g)\to\map(S_g)$ induces the isomorphism $\langle \G,\F\rangle \cong \langle G,F\rangle$.
\end{theorem}

\begin{proof}
First, assume that at least one of $F$ or $G$ is pseudo-Anosov. From Theorem~\cite[Theorem 3.1]{inf_meta_ggd}, it follows that the only possible case is that $F$ is pseudo-Anosov and $G$ is periodic satisfying $GFG^{-1}=F^{\epsilon}$, where $\epsilon=\pm 1$. Let $\G',\F\in \Diff(S_g)$ be such that $\pi(\G')=G$ and $\pi(\F)=F$, where $|\G|=|G|=m$ and $\F$ is pseudo-Anosov. Since $\pi(\G'\F(\G')^{-1})=GFG^{-1}=F^{\epsilon}=\pi(\F^{\epsilon})$, we have that $\G'\F(\G')^{-1}$ is isotopic to $\F^{\epsilon}$. Since $\G'\F(\G')^{-1}$ and $\F^{\epsilon}$ are both pseudo-Anosov homeomorphisms, by Theorem~\ref{thm:bonahon_flp} we get $\F^{\epsilon}=\h\G'\F(\G')^{-1}\h^{-1}$, where $\h$ is isotopic to the identity. For $\G:=\h\G'$, we have that $\G\F\G^{-1}=\F^{\epsilon}$ and $\pi(\G)=G$. Since $\pi(\G^m)=\pi(\G)^m=G^m=1$, it follows that $\G$ cannot be pseudo-Anosov. Since $\G^2$ commutes with $\F$, by Theorem~\ref{thm:chilina} $\G$ must be of finite order. Since $\G^m$ is isotopic to the identity, it follows from Theorem~\ref{thm:bonahon_flp} that $\G^m=1$. Hence, it follows that $\langle \G,\F\rangle$ is isomorphic to $\langle G,F\rangle$ induced by $\pi$.

Now, assume that both $G$ and $F$ are reducible mapping classes. It follows from~\cite[Theorem 3.12]{inf_meta_ggd} that every path component of $S_g(C(F) \cup C(G))$ either supports a finite metacyclic group or an infinite metacyclic group with at least one pseudo-Anosov generator. Due to Kerckhoff~\cite{kerckhoff_annals_1983} and the previous paragraph, these metacyclic components are realizable as a group of homeomorphisms. Thus, the group $\langle G,F \rangle$ lifts to an isomorphic subgroup $\langle \G,\F\rangle $ of $\Diff(S_g)$.
\end{proof}

We conclude the paper by proving the generalized Nielsen realization for the Baumslag-Solitar subgroups of $\map(S_g)$.

\begin{theorem}
\label{thm:nr_bsmn}
For $g\geq 2$ and $F,G\in \map(S_g)$, let $\BS(p,p\epsilon)=\langle G, F \mid GF^pG^{-1}=F^{p\epsilon} \rangle$ be a Baumslag-Solitar subgroup of $\map(S_g)$. Then, there exist $\G,\F\in \Diff(S_g)$ such that the natural projection map $\pi:\Diff(S_g)\to\map(S_g)$ induces the isomorphism $
\langle \G,\F \mid \G\F^p\G^{-1}=\F^{p\epsilon}\rangle \cong \BS(p,p\epsilon)$.
\end{theorem}

\begin{proof}
By Theorem~\ref{thm:main_theorem}, each path component $R$ of $S_g(\C(F))$ supports a subgroup
\[
\langle F_r,G_r\mid G_rF_r^pG_r^{-1}=F_r^{p\epsilon^{q_r}}, F_r^{u_r}=G_r^{v_r}=1 \rangle,
\] 
generated by the induced maps $F_r,G_r$ of $F,G$, respectively, where $q_r$ is the size of the orbit of $R$ under $G$. Since $\langle G_r,F_r^p \rangle$ is metacyclic, Theorem~\ref{thm:nr_inf_meta} gives Nielsen-Thurston-type-preserving homeomorphisms $\h_r,\G_r'\in \Diff(R)$ such that $\pi(\G_r')=G_r$, $\pi(\h_r)=F_r^p$, and $\G_r'\h_r(\G_r')^{-1}=\h_r^{\epsilon^{q_r}}$. Let $\F_r\in \Diff(R)$ be Nielsen-Thurston-type-preserving with $\pi(\F_r)=F_r$. Since $\pi(\F_r^p)=F_r^p=\pi(\h_r)$, Theorem~\ref{thm:bonahon_flp} gives $\G_r''\in \Diff(R)$ isotopic to the identity such that $\G_r''\F_r^p(\G_r'')^{-1}=\h_r$. For $\G_r:=(\G_r'')^{-1}\G_r'\G_r''$, which is isotopic to $\G_r'$, we have $\G_r\F_r^p\G_r^{-1}=\F_r^{p\epsilon^{q_r}}$. Since we can realize each induced subgroup supported on $R$ by an isomorphic subgroup of $\Diff(R)$, it follows that $\BS(p,p\epsilon)$ can also be realized as an isomorphic subgroup of $\Diff(S_g)$.
\end{proof}

\bibliographystyle{abbrv}
\bibliography{bsmn}

@incollection {farb_problems_book,
    AUTHOR = {Farb, Benson},
     TITLE = {Some problems on mapping class groups and moduli space},
 BOOKTITLE = {Problems on mapping class groups and related topics},
    SERIES = {Proc. Sympos. Pure Math.},
    VOLUME = {74},
     PAGES = {11--55},
 PUBLISHER = {Amer. Math. Soc., Providence, RI},
      YEAR = {2006},
      ISBN = {978-0-8218-3838-9; 0-8218-3838-5},
   MRCLASS = {57M50 (20F65 30F40 37D20)},
  MRNUMBER = {2264130},
       DOI = {10.1090/pspum/074/2264130},
       URL = {https://doi.org/10.1090/pspum/074/2264130},
}

@article {mark,
    AUTHOR = {Markovic, Vladimir},
     TITLE = {Realization of the mapping class group by homeomorphisms},
   JOURNAL = {Invent. Math.},
  FJOURNAL = {Inventiones Mathematicae},
    VOLUME = {168},
      YEAR = {2007},
    NUMBER = {3},
     PAGES = {523--566},
      ISSN = {0020-9910,1432-1297},
   MRCLASS = {57M60 (57S05)},
  MRNUMBER = {2299561},
MRREVIEWER = {Rafael\ Oswaldo\ Ruggiero},
       DOI = {10.1007/s00222-007-0039-0},
       URL = {https://doi.org/10.1007/s00222-007-0039-0},
}

@incollection {mann_gnr,
    AUTHOR = {Mann, Kathryn and Tshishiku, Bena},
     TITLE = {Realization problems for diffeomorphism groups},
 BOOKTITLE = {Breadth in contemporary topology},
    SERIES = {Proc. Sympos. Pure Math.},
    VOLUME = {102},
     PAGES = {131--156},
 PUBLISHER = {Amer. Math. Soc., Providence, RI},
      YEAR = {2019},
      ISBN = {978-1-4704-4249-1},
   MRCLASS = {57S25 (55R40 57M60)},
  MRNUMBER = {3967366},
MRREVIEWER = {Marja\ K.\ Kankaanrinta},
       DOI = {10.1090/pspum/102/11},
       URL = {https://doi.org/10.1090/pspum/102/11},
}

@article {min,
    AUTHOR = {Min, Honglin},
     TITLE = {Hyperbolic graphs of surface groups},
   JOURNAL = {Algebr. Geom. Topol.},
  FJOURNAL = {Algebraic \& Geometric Topology},
    VOLUME = {11},
      YEAR = {2011},
    NUMBER = {1},
     PAGES = {449--476},
      ISSN = {1472-2747,1472-2739},
   MRCLASS = {20F67 (20F28 20F65 57M07)},
  MRNUMBER = {2783234},
MRREVIEWER = {Igor\ Belegradek},
       DOI = {10.2140/agt.2011.11.449},
       URL = {https://doi.org/10.2140/agt.2011.11.449},
}

@article {bonahon_cobordism_1983,
    AUTHOR = {Bonahon, Francis},
     TITLE = {Cobordism of automorphisms of surfaces},
   JOURNAL = {Ann. Sci. \'Ecole Norm. Sup. (4)},
  FJOURNAL = {Annales Scientifiques de l'\'Ecole Normale Sup\'erieure.
              Quatri\`eme S\'erie},
    VOLUME = {16},
      YEAR = {1983},
    NUMBER = {2},
     PAGES = {237--270},
      ISSN = {0012-9593},
   MRCLASS = {57N05 (57N10)},
  MRNUMBER = {732345},
MRREVIEWER = {Klaus\ Johannson},
       URL = {http://www.numdam.org/item?id=ASENS_1983_4_16_2_237_0},
}

@article {chilina_2025,
    AUTHOR = {Chilina, E. E.},
     TITLE = {On the centralizer and conjugacy of pseudo-{A}nosov
              homeomorphisms},
   JOURNAL = {Russ. J. Nonlinear Dyn.},
  FJOURNAL = {Russian Journal of Nonlinear Dynamics},
    VOLUME = {21},
      YEAR = {2025},
    NUMBER = {1},
     PAGES = {103--116},
      ISSN = {2658-5324,2658-5316},
   MRCLASS = {37E30 (37C15 37D20 57R30)},
  MRNUMBER = {4893632},
       DOI = {10.20537/nd250301},
       URL = {https://doi.org/10.20537/nd250301},
}

@article {humphries_gms_1989,
    AUTHOR = {Humphries, Stephen P.},
     TITLE = {Free products in mapping class groups generated by {D}ehn
              twists},
   JOURNAL = {Glasgow Math. J.},
  FJOURNAL = {Glasgow Mathematical Journal},
    VOLUME = {31},
      YEAR = {1989},
    NUMBER = {2},
     PAGES = {213--218},
      ISSN = {0017-0895,1469-509X},
   MRCLASS = {57N05 (20E06 20F28 32G15)},
  MRNUMBER = {997819},
MRREVIEWER = {J.\ S.\ Birman},
       DOI = {10.1017/S001708950000776X},
       URL = {https://doi.org/10.1017/S001708950000776X},
}

@article {kerckhoff_annals_1983,
    AUTHOR = {Kerckhoff, Steven P.},
     TITLE = {The {N}ielsen realization problem},
   JOURNAL = {Ann. of Math. (2)},
  FJOURNAL = {Annals of Mathematics. Second Series},
    VOLUME = {117},
      YEAR = {1983},
    NUMBER = {2},
     PAGES = {235--265},
      ISSN = {0003-486X},
   MRCLASS = {32G15 (30F10 57M99 57N10)},
  MRNUMBER = {690845},
MRREVIEWER = {William Harvey},
       DOI = {10.2307/2007076},
       URL = {https://doi.org/10.2307/2007076},
}

@article {thurston_bams_1988,
    AUTHOR = {Thurston, William P.},
     TITLE = {On the geometry and dynamics of diffeomorphisms of surfaces},
   JOURNAL = {Bull. Amer. Math. Soc. (N.S.)},
  FJOURNAL = {American Mathematical Society. Bulletin. New Series},
    VOLUME = {19},
      YEAR = {1988},
    NUMBER = {2},
     PAGES = {417--431},
      ISSN = {0273-0979},
   MRCLASS = {57M99},
  MRNUMBER = {956596},
MRREVIEWER = {N. V. Ivanov},
       DOI = {10.1090/S0273-0979-1988-15685-6},
       URL = {https://doi.org/10.1090/S0273-0979-1988-15685-6},
}

@article {inf_meta_ggd,
    AUTHOR = {Kapari, Pankaj and Rajeevsarathy, Kashyap and Sanghi, Apeksha},
     TITLE = {Infinite metacyclic subgroups of the mapping class group},
   JOURNAL = {Groups Geom. Dyn.},
  FJOURNAL = {Groups, Geometry, and Dynamics},
    VOLUME = {19},
      YEAR = {2025},
    NUMBER = {1},
     PAGES = {281--313},
      ISSN = {1661-7207,1661-7215},
   MRCLASS = {57M60 (57K20)},
  MRNUMBER = {4862334},
       DOI = {10.4171/ggd/791},
       URL = {https://doi.org/10.4171/ggd/791},
}

@incollection {leininger_survey,
    AUTHOR = {Kent, IV, Richard P. and Leininger, Christopher J.},
     TITLE = {Subgroups of mapping class groups from the geometrical
              viewpoint},
 BOOKTITLE = {In the tradition of {A}hlfors-{B}ers. {IV}},
    SERIES = {Contemp. Math.},
    VOLUME = {432},
     PAGES = {119--141},
 PUBLISHER = {Amer. Math. Soc., Providence, RI},
      YEAR = {2007},
      ISBN = {978-0-8218-4227-0; 0-8218-4227-7},
   MRCLASS = {57M60 (20F34 20H10 30F60)},
  MRNUMBER = {2342811},
MRREVIEWER = {Natalia\ Kopteva},
       DOI = {10.1090/conm/432/08306},
       URL = {https://doi.org/10.1090/conm/432/08306},
}

@book {primer,
    AUTHOR = {Farb, Benson and Margalit, Dan},
     TITLE = {A primer on mapping class groups},
    SERIES = {Princeton Mathematical Series},
    VOLUME = {49},
 PUBLISHER = {Princeton University Press, Princeton, NJ},
      YEAR = {2012},
     PAGES = {xiv+472},
      ISBN = {978-0-691-14794-9},
   MRCLASS = {57M50 (20F36 20F65 57M07 57N05)},
  MRNUMBER = {2850125},
MRREVIEWER = {Stephen P. Humphries},
}
\end{document}